\documentclass[12pt,reqno]{amsart}
\usepackage[a4paper,bindingoffset=0.1in,left=1in,right=1in,top=1in,bottom=1in,footskip=.25in]{geometry}
\usepackage{indentfirst,amssymb,amsmath,amsthm}    
\usepackage{newtxtext,newtxmath}
\usepackage{setspace}
\usepackage{times}
\usepackage[utf8]{inputenc}
\usepackage[T1]{fontenc}
\usepackage{verbatim}
\usepackage{hyperref}
\hypersetup{colorlinks=true, linkcolor=blue,citecolor=blue, urlcolor=blue}
\DeclareMathOperator{\ord}{ord}
\DeclareMathOperator{\dime}{dim}
\DeclareMathOperator{\pideg}{PI-deg}

\DeclareMathOperator{\gcdi}{gcd}
\DeclareMathOperator{\lcmu}{lcm}
\DeclareMathOperator{\diagonal}{diag}

\DeclareMathOperator{\kere}{ker}
\DeclareMathOperator{\ran}{rank}

\usepackage{mathtools}

\numberwithin{equation}{section}
 
\newtheorem{theo}{Theorem}[section]
\newtheorem{defi}[theo]{Definition}
\newtheorem{lemm}[theo]{Lemma}
\newtheorem{rema}[theo]{Remark}

\newtheorem{prop}[theo]{Proposition}

\begin{document}

\setcounter{page}{1} 
\baselineskip .65cm 
\pagenumbering{arabic}

\title[Quantized Weyl Algebras]{SIMPLE MODULES OVER QUANTIZED WEYL ALGEBRAS\\ AT ROOTS OF UNITY}
\author [Sanu Bera~ And~ Snehashis Mukherjee]{Sanu Bera$^1$ \and Snehashis Mukherjee$^2$}

\address {\newline Sanu Bera$^1$ and Snehashis Mukherjee$^2$
\newline School of Mathematical Sciences, \newline Ramakrishna Mission Vivekananda Educational and Research Institute (rkmveri), \newline Belur Math, Howrah, Box: 711202, West Bengal, India.
 }
\email{\href{mailto:sanubera6575@gmail.com}{sanubera6575@gmail.com$^1$};\href{mailto:tutunsnehashis@gmail.com}{tutunsnehashis@gmail.com$^2$}}

\subjclass[2020]{16D60, 16D70, 16S85}
\keywords{Quantum Weyl Algebra, Simple modules, Polynomial Identity algebra}
\begin{abstract}
In this article the simple modules over the rank-two quantized Weyl algebras at roots of unity over an algebraically closed field are classified.
\end{abstract}
\maketitle
%\tableofcontents
\section{{Introduction}}
Let $\mathbb{K}$ be a field and $\mathbb{K}^*$ denotes the multiplicative group of non-zero elements of $\mathbb{K}$ and $n$ be a positive integer. Let $\Lambda:=\left(\lambda_{ij}\right)_{n \times n}$ be an  $n \times n$ multiplicatively antisymmetric matrix over $\mathbb{K}$, that is, $ \lambda_{ii}=1$ and $\lambda_{ij}\lambda_{ji}=1$ for all $1 \leq i,j\leq n$ and let $\underline{q}:=(q_1,\cdots,q_n)$ be an $n$-tuple elements of $\mathbb{K}\setminus\{0,1\}$.
\begin{defi} (Quantum Weyl Algebra)
Given such $\Lambda$ and $\underline{q}$, the multiparameter quantum Weyl algebra ${A_n^{\underline{q},\Lambda}}$  of rank $n$ is the algebra generated over the field $\mathbb{K}$ by the variables $x_1,y_1,\cdots ,x_n,y_n$ subject to the following relations: 
\[\begin{array}{ll}
x_ix_j=q_i\lambda_{ij}x_jx_i,\ \  x_iy_j=\lambda_{ij}^{-1}y_jx_i,& 1\leq i<j\leq n\\
y_iy_j=\lambda_{ij}y_jy_i,\ \   y_ix_j=q_i^{-1}\lambda_{ij}^{-1}x_jy_i,& 1\leq i<j\leq n\\
x_iy_i-q_iy_ix_i=1+\sum_{k=1}^{i-1}(q_k-1)y_kx_k & 1\leq i\leq n.
\end{array}\]
\end{defi}
The quantum Weyl algebra ${A_n^{\underline{q},\Lambda}}$, which arises from the work of Maltsiniotis on noncommutative differential calculus \cite{gm}, has been extensively studied in \cite{ad,gz,krg,daj}. Another family of multiparameter quantum Weyl algebras has been studied in the literature including \cite{aj}, which has more symmetric defining relations than those of ${{A}_n^{\underline{q},\Lambda}}$. 
\begin{defi} (Alternative Quantum Weyl Algebra)
Given such $\Lambda$ and $\underline{q}$, the alternative multiparameter quantum Weyl algebra ${\mathcal{A}_n^{\underline{q},\Lambda}}$ of rank $n$ is the algebra generated over the field $\mathbb{K}$ by the variables $x_1,y_1,\cdots ,x_n,y_n$ subject to the following relations:
\[\begin{array}{ll}
x_ix_j=\lambda_{ij}x_jx_i,\ \  x_iy_j=\lambda_{ij}^{-1}y_jx_i,& 1\leq i<j\leq n\\
y_iy_j=\lambda_{ij}y_jy_i,\ \   y_ix_j=\lambda_{ij}^{-1}x_jy_i,& 1\leq i<j\leq n\\
x_iy_i-q_iy_ix_i=1 & 1\leq i\leq n.
\end{array}\]
\end{defi}
Here the adjective `quantized' will refer to both versions of quantum Weyl algebras simultaneously. We note that the algebras ${A_n^{\underline{q},\Lambda}}$ and ${\mathcal{A}_n^{\underline{q},\Lambda}}$ are closely related in many aspects \cite{aj, daj}. Both algebras have some iterated skew polynomial presentation twisted by automorphisms and derivations and have a common localization (see section 2). Analogous to the Weyl algebra, these quantized versions can be viewed as algebras of partial $q$-difference operators on quantum affine spaces (see \cite{aj}). The prime spectrum, the automorphism group, and the isomorphism problem for both versions of quantum Weyl algebras were studied in \cite{krg,lr,gh2,aj,xt}. Most of these results concern the generic case when the algebras are not polynomial identity. In this article, we shall focus on these algebras when they are PI-algebras.
\par The quantum Weyl algebras have appeared in numerous works in mathematics and physics, including deformation theory, knot theory, category theory, and quantum mechanics to name a few. In quantum mechanics, the algebra of observables is noncommutative and the objects which play the role of points are the irreducible representations of the algebra of observables. Hence, it is natural to understand the irreducible representations of quantum Weyl algebras.   
\par For an infinite dimensional noncommutative algebra, classifying its simple modules is a very difficult problem, in general. When $n=1$, both quantum Weyl algebras coincide with rank-one quantum Weyl algebra $A_1^{q}$ with generators $x,y$ and relation $xy-qyx=1$. If $q$ is a root of unity, all irreducible representations of $A_1^{q}$ are finite-dimensional and were first classified in \cite{dgo,dj,bav} and later explicitly described them up to equivalence in two research projects directed by E. Letzter \cite{blt} and by L. Wang \cite{he} via studying the matrix solutions $(X,Y)$ of the equation $xy-qyx=1$. However, to date, no classification of simple modules over the quantized Weyl algebras of rank two onward has appeared in the literature. The purpose of this paper is to classify simple modules over the rank-two quantized Weyl algebras at roots of unity i.e., in the case when they are PI.
\subsection*{Assumptions} 
For $n=2$ we will use $\lambda=\lambda_{12}$ in the definition of the quantized Weyl algebras ${A_2^{\underline{q},\Lambda}}$ and ${\mathcal{A}_2^{\underline{q},\Lambda}}$. Throughout this article we shall assume the following assumption on the multiparameters $\lambda$ and $\underline{q}$ in the root of unity setting:
\begin{equation}\label{asm}
\begin{minipage}{0.9\textwidth}
\begin{itemize}
    \item $q_i$ is primitive $l_i$-th roots of unity for all $i=1,2$ and 
    \item $\lambda$ is $l_1$-th roots of unity.
    \end{itemize}
\end{minipage}\tag{*}
\end{equation}
These assumptions are satisfied in the important uniparameter case when $q_1=q_2$ and $\lambda_{12}=1$. Under this assumption the multiplicative group $\Gamma$ generated by the multiparameters $q_1,q_2$ and $\lambda$ is a cyclic group of order $l=\lcmu(l_1,l_2)$. Throughout this paper $\mathbb{K}$ is an algebraically closed field of arbitrary characteristics and all modules are the right modules.
\subsection*{Arrangement:} 
The paper is organized as follows. In Section $2$ we review some necessary facts for quantized Weyl algebras and discuss the theory of Polynomial Identity algebras to comment about the $\mathbb{K}$-dimension of the simple modules over such algebras. In Section $3$ an explicit expression of PI-degree for such quantized algebras is computed under the assumption (\ref{asm}) on multiparameter. Finally, in Section $4$ and Section $5$, the simple modules over the quantized Weyl algebras are classified.
\section{\bf{Preliminaries}}
In this section, we recall some important facts for quantized Weyl algebras that we shall be applying to compute the PI degree and classify simple modules. 
\subsection{\bf{Quantum Weyl Algebra}}
Here we recall some necessary facts for the quantum Weyl algebra in particular when $n=2$. The algebra ${{A}_2^{\underline{q},\Lambda}}$ has an iterated skew polynomial presentation with respect to the ordering of variables $y_1,x_1,y_2,x_2$ of the form: 
$$\mathbb{K}[y_1][x_1,\tau_1,\delta_1][y_2,\sigma_2][x_2,\tau_2,\delta_2]$$
where the $\tau_1,\tau_2$ and $\sigma_{1}$ are $\mathbb{K}$-linear automorphisms and the $\delta_j$ are $\mathbb{K}$-linear $\tau_j$-derivations  such that
\[\tau_1(y_1)=q_1 y_1,\ \delta_1(y_1)=1,\ \sigma_2(y_1)=\lambda^{-1}y_1,\ \sigma_2(x_1)=\lambda x_1\]
\[\tau_2(y_1)=q_1\lambda y_1,\ \tau_2(x_1)=(q_1\lambda)^{-1} x_1,\ \tau_2(y_2)=q_2y_2,\]
\[\delta_2(x_1)=\delta_2(y_1)=0,\ \delta_2(y_2)=1+(q_1-1)y_1x_1\]
This observation along with the skew polynomial version of the Hilbert Basis Theorem (cf. \cite[Theorem 2.9]{mcr}) yields that the algebra ${A_2^{\underline{q},\Lambda}}$ is an affine noetherian domain and has a $\mathbb{K}$-basis $\{y_1^{a_1}x_1^{b_1}y_2^{a_2}x_2^{b_2}:a_i,b_i\in \mathbb{Z}_{\geq 0}\}$, see \cite[1.9]{aj}. 
\par Let us define $z_0=1$ and $z_i=x_iy_i-y_ix_i$ for $i=1,2$. Using the defining relations of the algebra ${{A}_2^{\underline{q},\Lambda}}$, one can easily verified the following: 
\begin{lemm}\emph{(\cite[1.4]{aj})}\label{crq}
Direct computations yield the following commutation relations and identities.
\begin{itemize}
    \item $z_2=1+(q_1-1)y_1x_1+(q_2-1)y_2x_2=z_1+(q_2-1)y_2x_2$ and $x_2y_2-q_2y_2x_2=z_1$.
    \item $z_ix_j=q_j^{-1}x_jz_i,z_iy_j=q_jy_jz_i,~1\leq j\leq i\leq 2$ and $z_1x_2=x_2z_1,z_1y_2=y_2z_1,z_1z_2=z_2z_1$.
    \item $x_i^ny_i=q_i^ny_ix_i^n+\left(1+q_i+\cdots+q_i^{n-1}\right)z_{i-1}x_i^{n-1}$ and \\
 $x_iy_i^n=q_i^ny_i^nx_i+\left(1+q_i+\cdots+q_i^{n-1}\right)z_{i-1}y_i^{n-1}$ for $i=1,2$.
\end{itemize}
\end{lemm}
These commutation relations are going to play a crucial role throughout the article. Recall that a nonzero element $x$ of an algebra ${A}$ is called a normal element if $x{A}={A}x$. Clearly the elements $z_1,z_2$ are normal in ${A_2^{\underline{q},\Lambda}}$.
\subsection{Alternative Quantum Weyl Algebra} Here we also recall some known facts for rank-two alternative quantum Weyl algebra which are analogous to quantum Weyl algebra. The algebra ${\mathcal{A}_2^{\underline{q},\Lambda}}$ has an iterated skew polynomial presentation with respect to the ordering of variables $y_1,x_1,y_2,x_2$ of the form:
$$\mathbb{K}[y_1][x_1,\tau_1,\delta_1][y_2,\sigma_2][x_2,\tau_2,\delta_2]$$
where the $\tau_1,\tau_2$ and $\sigma_{1}$ are $\mathbb{K}$-linear automorphisms and the $\delta_j$ are $\mathbb{K}$-linear $\tau_j$-derivations  such that
\[\tau_1(y_1)=q_1 y_1,\ \delta_1(y_1)=1,\ \sigma_2(y_1)=\lambda^{-1}y_1,\ \sigma_2(x_1)=\lambda x_1\]
\[\tau_2(y_1)=\lambda y_1,\ \tau_2(x_1)=\lambda^{-1} x_1,\ \tau_2(y_2)=q_2y_2,\ \delta_2(x_1)=\delta_2(y_1)=0,\ \delta_2(y_2)=1.\]
This observation yields that the algebra ${\mathcal{A}_2^{\underline{q},\Lambda}}$ is an affine noetherian domain and has a $\mathbb{K}$-basis $\{y_1^{a_1}x_1^{b_1}y_2^{a_2}x_2^{b_2}:a_i,b_i\in \mathbb{Z}_{\geq 0}\}$. 
\par Let us also define $z_i=x_iy_i-y_ix_i$ for $i=1,2$. As before, these elements will play a crucial role. The following results are analogous to Lemma \ref{crq}.
\begin{lemm}\emph{(\cite[1.5]{aj})}\label{craq} The following commutation relations and identities hold in ${\mathcal{A}_2^{\underline{q},\Lambda}}$:
\begin{itemize}
    \item $z_i=1+(q_i-1)y_ix_i$ for $i=1,2$ and $z_1z_2=z_2z_1$.
    \item $z_ix_i=q_i^{-1}x_iz_i,z_iy_i=q_iy_iz_i$ and $z_ix_j=x_jz_i,z_iy_j=y_jz_i,~1\leq i\neq j\leq 2$.
    \item $x_i^ny_i=q_i^ny_ix_i^n+\left(1+q_i+\cdots+q_i^{n-1}\right)x_i^{n-1}$ and \\
 $x_iy_i^n=q_i^ny_i^nx_i+\left(1+q_i+\cdots+q_i^{n-1}\right)y_i^{n-1}$ for $i=1,2$.
\end{itemize}
\end{lemm}
Note that the elements $z_1$ and $z_2$ are normal in ${\mathcal{A}_2^{\underline{q},\Lambda}}$ with more simpler commutation relations than those for ${{A}_2^{\underline{q},\Lambda}}$. 
\subsection{\bf{A Common Localization}} Here we recall a result from \cite{aj} which states that the localizations of both versions of quantum Weyl algebras are isomorphic. The multiplicative set generated by the $z_1$ and $z_2$ consists of normal elements. Hence it is an Ore set and thus we may form the corresponding localization. We denote by ${{B}_2^{\underline{q},\Lambda}}$ the localization of ${{A}_2^{\underline{q},\Lambda}}$ by inverting the normal elements $z_1,z_2$. We denote the corresponding localization of ${\mathcal{A}_2^{\underline{q},\Lambda}}$ by ${\mathcal{B}_2^{\underline{q},\Lambda}}$.
\begin{theo}\emph{(\cite[1.7]{aj})}
The map $\theta:{\mathcal{B}_2^{\underline{q},\Lambda}}\longrightarrow{{B}_2^{\underline{q},\Lambda}}$ defined by $y_i\mapsto y_i, x_i\mapsto z_{i-1}^{-1}x_i$ for $i=1,2$ with $z_0=1$
is an isomorphism of $\mathbb{K}$-algebras.
\end{theo}
The above theorem yields an important link between the simple modules over both versions of quantum Weyl algebras. Note that the action of a normal element $x$ of an algebra $A$ on a simple $A$-module $M$ is either zero (if $M$ is $x$-torsion) or invertible (if $M$ is $x$-torsionfree).
\begin{prop}\label{conn}
There is a one to one correspondence between the simple $z_1,z_2$-torsionfree ${{A}_2^{\underline{q},\Lambda}}$-modules and the simple $z_1,z_2$-torsionfree ${\mathcal{A}_2^{\underline{q},\Lambda}}$-modules.
\end{prop}
\subsection{\bf{Polynomial Identity Algebras}}
Here we wish to establish some facts for quantized Weyl algebras at the roots of unity related to polynomial identity algebras. The following result provides a necessary and sufficient condition for quantized Weyl algebras to be PI.
\begin{prop} \label{finite}
The quantum Weyl algebra ${A_2^{\underline{q},\Lambda}}$ (respectively,  ${\mathcal{A}_2^{\underline{q},\Lambda}}$) is a PI algebra if and only if the parameters $q_1,q_2$ and $\lambda$ are roots of unity. 
\end{prop}
The first part is a consequence of Kaplansky's Theorem (cf. \cite[ Theorem 13.3.8]{mcr}). For the converse part, use the fact that if an algebra is a finitely generated module over its center then it is PI. 
\par Kaplansky's Theorem has a striking consequence in the case of a prime affine PI algebra over an algebraically closed field. Let $A$ be a prime affine PI algebra over an algebraically closed field $\mathbb{K}$ and $V$ be a simple $A$-module. Then by \cite[Theorem I.13.5]{brg} we have $\dime_{\mathbb{K}}(V)\leq \pideg (A)$. This result provides the important link between the PI degree of a prime affine PI algebra over an algebraically closed field and its irreducible representations. Moreover, the upper bound PI-deg($A$) is in fact attained for such an algebra $A$ (cf. \cite[Lemma III.1.2]{brg}). 
\begin{rema}\label{fdb}
From the above discussion, it is quite clear that each simple ${A_2^{\underline{q},\Lambda}}$-module is finite-dimensional and can have dimension at most $\pideg( {A_2^{\underline{q},\Lambda}}$). The same is true for the algebra ${\mathcal{A}_2^{\underline{q},\Lambda}}$ also.
\end{rema}
In the following section, we will focus on computing the PI degrees of the concerned algebras.
\section{\bf{PI Degree for Quantized Weyl algebras}}
Let us first recall the assumption (\ref{asm}) on multiparameters. Then the multiplicative group $\Gamma$ generated by the multiparameters $q_1,q_2$ and $\lambda$ is a cyclic group of order $l=\lcmu(l_1,l_2)$ and hence the quantized Weyl algebras are PI algebras. In this section, we aim to compute an explicit expression of PI-degree for quantized Weyl algebras with the assumption (\ref{asm}). Here we will use the derivation erasing process \cite[Theorem 7]{lm2} and then a key technique for calculating PI degree of a quantum affine space \cite[Proposition 7.1]{di}. We denote the quantum affine space of rank $n$ by $\mathcal{O}_{\Lambda}(\mathbb{K}^n)$, the $\mathbb{K}$-algebra generated by the variables $x_1,\cdots ,x_n$ subject to the relations \[x_ix_j=\lambda_{ij}x_jx_i, \ \ \ \forall\ \ \ 1 \leq i,j\leq n.\]
\subsection{\bf{PI Degree for ${{A}_2^{\underline{q},\Lambda}}$}} First recall that the quantum Weyl algebra ${{A}_2^{\underline{q},\Lambda}}$ has an iterated skew polynomial algebra twisted by automorphisms and derivations (see subsection 2.1). With this, we are ready to use the derivation erasing independent of characteristics due to A. Leroy and J. Matczuk \cite{lm2}.\\
\textbf{Step 1:} (Derivation Erasing) A straightforward calculation shows that the twisting $\mathbb{K}$-linear maps $\tau_j$ and $\delta_j$ of ${{A}_2^{\underline{q},\Lambda}}$ satisfy the skew relations $\delta_j\tau_j=q_j\tau_j\delta_j$ $(q_j\neq 1)$ for $j=1,2$. Also one can check that all the hypothesis of derivation erasing process in \cite[Theorem 7]{lm2} is satisfied by the PI algebra ${{A}_2^{\underline{q},\Lambda}}$. Hence it follows that $\pideg {{A}_2^{\underline{q},\Lambda}}= \pideg \mathcal{O}_{\mathbf{q}}(\mathbb{K}^{4})$, where the $(4\times 4)$-matrix of relations $\mathbf{q}$ is
\begin{equation}\label{defma}
\mathbf{q}=\begin{pmatrix}
1&q_1^{-1}&\lambda &q_1^{-1}\lambda^{-1}\\
q_1&1&\lambda^{-1}&q_1\lambda\\
\lambda^{-1}&\lambda&1&q_2^{-1}\\
q_1\lambda& q_1^{-1}\lambda^{-1}&q_2&1
\end{pmatrix}.
\end{equation}
\noindent\textbf{Step 2:} (Integral matrix) Now we wish to form an integral matrix associated with $\mathbf{q}$ under the assumption (\ref{asm}). Let $q$ be a generator of the group $\Gamma$. Then we can choose $0<s_i<l$ with $s_i$ divides $l$ such that $\langle q_i\rangle=\langle q^{s_i}\rangle$ for $i=1,2$. Note that $\gcdi(s_1,s_2)=1$ as $s_1l_1=l=s_2l_2$. Therefore there exists integers $k_1,k_2$ and $k$ such that 
\[q_1=q^{s_1k_1},\ q_2=q^{s_2k_2},\ \lambda=q_1^k=q^{s_1k_1k}.\]
Also note that $\gcdi(k_1,l_1)=\gcdi(k_2,l_2)=1$. Thus the skew-symmetric integral matrix associated to $\mathbf{q}$ is
\[B=\begin{pmatrix}
0&-s_1k_1&s_1k_1k&-s_1k_1-s_1k_1k\\
s_1k_1&0&-s_1k_1k&s_1k_1+s_1k_1k\\
-s_1k_1k&s_1k_1k&0&-s_2k_2\\
s_1k_1+s_1k_1k&-s_1k_1-s_1k_1k&s_2k_2&0
\end{pmatrix}.\] One can easily check that this matrix is similar to an integral matrix 
\[C=\diagonal\left(\begin{pmatrix} 
0 & -s_1k_1 \\
s_1k_1 & 0
\end{pmatrix},\begin{pmatrix} 
0 & -s_2k_2 \\
s_2k_2 & 0
\end{pmatrix}\right)\] and hence they have the same invariant factors with $\ran(B)=\ran(C)=4$. \\
\textbf{Step 3:} (Invariant factors) Suppose that $h_1|h_1|h_2|h_2$ are the invariant factors for $C$ as well as for $B$. Now from the relations between invariant factors and determinantal divisors, it follows that
\[\begin{array}{ll}
h_1&=\text{first~determinantal~divisor}=\gcdi(s_1k_1,s_2k_2)\ \text{and}\\
h_2&=\displaystyle\frac{\text{fourth~determinantal~divisor}}{\text{third~determinantal~divisor}}\\
&=\displaystyle\frac{(s_1k_1s_2k_2)^2}{\gcdi(s_1k_1(s_2k_2)^2,(s_1k_1)^2s_2k_2)}=\lcmu(s_1k_1,s_2k_2).
\end{array}\]
In the next step we will use these expressions of $h_1$ and $h_2$. Note that $h_1h_2=s_1k_1s_2k_2$.\\
\noindent\textbf{Step 3:} (Final step) One of the key techniques for calculating the PI degree of a quantum affine space was first introduced in \cite[Proposition 7.1]{di} and later simplified this with respect to invariant factors in \cite[Lemma 5.7]{ar}. Therefore applying the result \cite[Lemma 5.7]{ar} in our context we have
\begin{equation}\label{pid}
  \pideg(\mathcal{O}_{\mathbf{q}}(\mathbb{K}^4))=\frac{l}{\gcdi(h_1,l)}\times \frac{l}{\gcdi(h_2,l)}.  
\end{equation}
Finally we claim that $\gcdi(h_1,l)=1$ and $\gcdi(h_2,l)=s_1s_2$. 
\par To prove the first claim suppose $\gcdi(h_1,l)=d$. Then $d$ divides $s_1k_1,s_2k_2$ and $l$. Now we compute
\[q_1^{\frac{l}{d}}=q^{\frac{s_1k_1l}{d}}=1\ \ \text{and}\ \ q_2^{\frac{l}{d}}=q^{\frac{s_2k_2l}{d}}=1.\] This implies that $l_1$ and $l_2$ both divides $\frac{l}{d}$ and hence $l=\lcmu(l_1,l_2)$ divides $\frac{l}{d}$. Thus we can conclude that $\gcdi(h_1,l)=1$.
\par We now prove the second claim using the first claim. One can simplify the following using $\gcdi(k_1,l_1)=\gcdi(k_2,l_2)=1$: 
\[\gcdi(h_2,l)=\gcdi(h_1h_2,l)=s_1s_2\gcdi\left(k_1k_2,\frac{l}{s_1s_2}\right)=s_1s_2.\] Therefore we can simplify the equation (\ref{pid}) with these claims as follows: \[\pideg {{A}_2^{\underline{q},\Lambda}}=\pideg(\mathcal{O}_{\mathbf{q}}(\mathbb{K}^4))=\frac{l^2}{s_1s_2}=l_1l_2.\]
Thus we have proved that
\begin{theo}\label{pidq}
The PI-degree of quantum Weyl algebra ${{A}_2^{\underline{q},\Lambda}}$ is $l_1l_2$ under the assumption (\ref{asm}) on multiparameters.
\end{theo}
\subsection{PI Degree of ${\mathcal{A}_2^{\underline{q},\Lambda}}$} Here the computation of PI degree for ${\mathcal{A}_2^{\underline{q},\Lambda}}$ is analogous to that for ${{A}_2^{\underline{q},\Lambda}}$. First observe that the alternative quantum Weyl algebra ${\mathcal{A}_2^{\underline{q},\Lambda}}$ has an iterated skew polynomial presentation and satisfies all the hypotheses of derivation erasing process \cite[Theorem 7]{lm2}. Hence it follows that $\pideg {\mathcal{A}_2^{\underline{q},\Lambda}}= \pideg \mathcal{O}_{\mathbf{q}}(\mathbb{K}^{4})$, where the $(4\times 4)$-matrix of relations $\mathbf{q}$ is
\[\mathbf{q}=\begin{pmatrix}
1&q_1^{-1}&\lambda &\lambda^{-1}\\
q_1&1&\lambda^{-1}&\lambda\\
\lambda^{-1}&\lambda&1&q_2^{-1}\\
\lambda& \lambda^{-1}&q_2&1
\end{pmatrix}.\] Note that this matrix $\mathbf{q}$ slightly differs from the previous one for ${{A}_2^{\underline{q},\Lambda}}$. In fact, it can be easily checked that the skew-symmetric integral matrix associated with this $\mathbf{q}$ under the assumption (\ref{asm}) is in the same similarity class to the block diagonal matrix $C$ as mentioned in step $2$. Therefore repeating the above steps of the argument, we obtain the following:
\begin{theo}
The PI-degree of alternative quantum Weyl algebra ${\mathcal{A}_2^{\underline{q},\Lambda}}$ is $l_1l_2$ under the assumption (\ref{asm}) on multiparameters.
\end{theo}
\section{\bf{Simple modules over quantum Weyl algebra}}
In this section, we shall focus on classifying simple ${{A}_2^{\underline{q},\Lambda}}$-modules at roots of unity under the assumption (\ref{asm}). Let $M$ be a simple ${{A}_2^{\underline{q},\Lambda}}$-module. Then by Remark \ref{fdb} and Theorem \ref{pidq} we have the $\mathbb{K}$-dimension of $M$ is finite and bounded above by $l_1l_2$. Now it follows from the defining relations of the algebra and Lemma \ref{crq} that the elements $x_1^{l_1},y_1^{l_1},x_2^{l}$ and $y_2^{l}$ are central, where $l=\lcmu(l_1,l_2)$. By Schur's lemma, these central elements act as scalars on $M$. Also, we have mentioned that each $z_i$ is normal and satisfies some commutation relations with the generators $x_i$ and $y_i$ for $i=1,2$. Moreover $z_1$ commutes with $x_2,y_2$ and $z_2$ (see Lemma \ref{crq}).
\par Note that the action of a normal element $x$ of an algebra $A$ on a simple $A$-module $M$ is either zero (if $M$ is $x$-torsion) or invertible (if $M$ is $x$-torsionfree). Thus the action of each $z_i$ on $M$ is either zero or invertible. Now depending on this the classification reduces in the following cases: 
\subsection{Simple $z_1,z_2$-torsionfree ${{A}_2^{\underline{q},\Lambda}}$-modules}\label{sub1} 
Let $M$ be a $z_1,z_2$-torsionfree simple module over ${{A}_2^{\underline{q},\Lambda}}$. Since each of the elements 
\begin{equation}\label{cev1}
   x_1^{l_1},\ x_2^{l_2},\ y_1^{l_1},\ y_2^{l_2},\ z_1,\ z_2
\end{equation}
of ${{A}_2^{\underline{q},\Lambda}}$ commutes (by Lemma \ref{crq}), there is a common eigenvector $v$ in $M$ of the operators in (\ref{cev1}).
Take
\[
\begin{array}{ll}
vx_i^{l_i}=\alpha_iv, & i=1,2\\
vy_i^{l_i}=\beta_iv, & i=1,2\\
vz_i=\gamma_iv, & i=1,2
\end{array}
\]
for some $\alpha_i,\beta_i\in \mathbb{K}$ and $\gamma_i\in \mathbb{K}^*$. As the central elements act as scalars on $M$, then the operators $x_i$ (respectively, $y_i$) on $M$ is nilpotent if and only if $\alpha_i=0$ (respectively, $\beta_i=0$) for $i=1,2$. In the following we shall determine the structure of simple ${{A}_2^{\underline{q},\Lambda}}$-module $M$ according to the scalars $\alpha_i$ and $\beta_i$:\\
\subsubsection{\bf{First consider $\alpha_1\neq 0$ and $\alpha_2\neq 0$}}\label{base} The operators $x_1$ and $x_2$ are invertible on $M$. So the vectors 
\[e(a_1,a_2):=vx_2^{a_2}x_1^{a_1},\ 0\leq a_i\leq l_i-1\] in $M$ are non-zero. Suppose $M_1$ is the $\mathbb{K}$-subspace of $M$ spanned by these nonzero vectors. We now claim that $M_1$ is invariant under the action of ${{A}_2^{\underline{q},\Lambda}}$. In fact, after some straightforward calculations using the defining relations and the identities in Lemma \ref{crq}, we have
\[\begin{array}{l}
e(a_1,a_2)x_1=\begin{cases}
 e(a_1+1,a_2),& 0\leq a_1\leq l_1-2\\
 \alpha_1 e(0,a_2),& a_1=l_1-1
\end{cases}\\
e(a_1,a_2)x_2=\begin{cases}
 (q_1\lambda)^{a_1}e(a_1,a_2+1),& 0\leq a_2\leq l_2-2\\
 \alpha_2(q_1\lambda)^{a_1} e(a_1,0),& a_2=l_2-1
\end{cases}\\

e(a_1,a_2)y_1=\begin{cases}
 \displaystyle\frac{q_1^{a_1}\gamma_1-1}{q_1-1}e(a_1-1,a_2),& a_1\neq 0\\
 \alpha_1^{-1}\displaystyle\frac{\gamma_1-1}{q_1-1} e(l_1-1,a_2),& a_1=0
\end{cases}\\
e(a_1,a_2)y_2=\begin{cases}
 \lambda^{-a_1}\displaystyle\frac{q_2^{a_2}\gamma_2-\gamma_1}{q_2-1}e(a_1,a_2-1),& a_2\neq 0\\
 \alpha_2^{-1}\displaystyle\frac{\gamma_2-\gamma_1}{q_2-1}\lambda^{-a_1} e(a_1,l_2-1),& a_2=0
\end{cases}
\end{array}\] 
This shows that $M_1$ is a  ${{A}_2^{\underline{q},\Lambda}}$-submodule of $M$. Since $M$ is a simple module, therefore we have $M=M_1$. The simple ${{A}_2^{\underline{q},\Lambda}}$-module ${M}_1$ given above is denoted by $(M_1,\alpha_1,\alpha_2,\gamma_1,\gamma_2)$ for $\alpha_1,\alpha_2,\gamma_1,\gamma_2\in\mathbb{K}^*$. Now the following result deciphers the $\mathbb{K}$-dimension of $M_1$.\\  
\begin{theo}\label{gqa}
The simple ${{A}_2^{\underline{q},\Lambda}}$-module $(M_1,\alpha_1,\alpha_2,\gamma_1,\gamma_2)$ has dimension $l_1l_2$.
\end{theo}
\begin{proof}
First observe from the above action that $\dime_{\mathbb{K}}(M_1)\leq l_1l_2$. Now to show linear independence, we will use induction on the $k$-element subsets consisting of the nonzero vectors $e(a_1,a_2)$. Clearly, it is true for $k=1$. Let us assume that any such $k$-element subsets are independent. Suppose \[S:=\{e(a^{(i)}_1,a^{(i)}_2):i=1,\cdots,k+1\}\] be a set of $k+1$ vectors such that  
\begin{equation}\label{sum}
  p:=\sum_{i=1}^{k+1} \xi_i~e(a^{(i)}_1,a^{(i)}_2)=0
\end{equation} for some $\xi_i\in \mathbb{K}$. Note that $e(a_1^{(i)},a_2^{(i)})$ is an eigenvector of the operator $z_1$ and $z_2$ on $M_1$ with eigenvalue $\nu_1^{(i)}:=q_1^{a_1^{(i)}}\gamma_1$ and $\nu_{2}^{(i)}:=q_1^{a_1^{(i)}}q_2^{a_2^{(i)}}\gamma_2$ respectively. Now the vectors $e\left(a_1^{(k)},a_2^{(k)}\right)$ and $e\left(a_1^{(k+1)},a_2^{(k+1)}\right)$ are distinct so that these vectors are eigenvector corresponding to the operator $z_r$ with distinct eigenvalues say, $\nu_r^{(k)}$ and $\nu_{r}^{(k+1)}$ respectively for at least one $r=1,2$. With this $r$ we compute that
\[0=pz_r-\nu_{r}^{(k+1)}p=\sum_{i=1}^{k} \xi_i(\nu_r^{(i)}-\nu_r^{(k+1)})~e\left(a_1^{(i)},a_2^{(i)}\right).
\]
Hence by induction hypothesis, it follows that
\begin{center}
    $\xi_i(\nu_r^{(i)}-\nu_r^{(k+1)})=0$ for all $i=1,\cdots,k$.
\end{center}
As $\nu_r^{(k)} \neq \nu_r^{(k+1)}$, we have $\xi_k=0$. Put $\xi_k=0$ in (\ref{sum}) and again using induction hypothesis we have $\xi_i=0$ for all $i=1,\cdots,k+1$. Hence the result follows.
\end{proof}
\begin{rema}
The simple ${{A}_2^{\underline{q},\Lambda}}$ modules $(M_1,\alpha_1,\alpha_2,\gamma_1,\gamma_2)$ and $(M_1,\alpha'_1,\alpha'_2,\gamma'_1,\gamma'_2)$ are isomorphic if and only if $\alpha_1=\alpha'_1,\beta_2=\beta'_2,\gamma_1=q_1^a\gamma'_1$ and $\gamma_2=q_1^{a}q_2^{b}\gamma'_2$, for some $0\leq a\leq l_1-1$ and $0\leq b\leq l_2-1$.
\end{rema}
\subsubsection{\bf{Next consider $\alpha_1=0$ and $\alpha_2\neq 0$}}\label{sov} First observe that the operator $x_1$ is nilpotent and the operator $x_2$ is invertible on $M$. Denote $\kere (x_1):=\{m\in M:mx_1=0\}$. Therefore $\kere (x_1)$ is a non-zero subspace of $M$.\\
\textbf{Step 1:} Now each of the commuting operators in $(\ref{cev1})$ must keep the $\mathbb{K}$-space $\kere (x_1)$ invariant due to commutation relations with $x_1$. Therefore we can choose a common eigenvector $w\in \kere (x_1)$ of the commuting operators $(\ref{cev1})$. Take 
\[\begin{array}{l}
wx_1=0,\ wx^{l_2}_2=\xi_2w,\\
wy^{l_i}_i=\eta_iw,~i=1,2\\
wz_i=\zeta_iw,~i=1,2
\end{array}\]
Clearly $\xi_2,\zeta_1,\zeta_2\in\mathbb{K}^*$ and $\eta_1,\eta_2\in\mathbb{K}$.\\
\textbf{Step 2:} 
We now claim that the vectors $wy_1^{a},~0\leq a\leq l_1-1$ in $M$ are non-zero. If $\eta_1 \neq 0$ then we are done. Otherwise suppose $k$ be the smallest index with $1\leq k\leq l_1$ such that $wy_1^{k-1}\neq 0\ \text{and}\ wy_1^{k}=0$. In fact after simplifying the equality $wy_1^{k}x_1=0$ we have 
\[
0=wy_1^{k}x_1=q_1^{-k}w\left(x_1y_1^{k}-\displaystyle\frac{q_1^{k}-1}{q_1-1}y_1^{k-1}\right)=-q_1^{-k}\displaystyle\frac{q_1^{k}-1}{q_1-1}wy_1^{k-1}.\]
This implies $k$ is the smallest index such that $q_1^{k}=1$ and hence $k=l_1$. This completes the proof of claim.\\
\textbf{Step 3:} As $\xi_2\neq 0$, it follows that the vectors \[e(a_1,a_2):=wx_2^{a_2}y_1^{a_1},\ 0\leq a_i\leq l_i-1\] in $M$ are nonzero. Let $M_2$ be the $\mathbb{K}$-subspace of $M$ spanned by these nonzero vectors. Similar to the previous one, $M_2$ is invariant under the action of ${{A}_2^{\underline{q},\Lambda}}$. In fact, using the defining relations and the identities in Lemma \ref{crq}, the action of generators of ${{A}_2^{\underline{q},\Lambda}}$ on $M_2$ is given by 
\[\begin{array}{l}
e(a_1,a_2)x_1=\begin{cases}
\displaystyle\frac{q_1^{-a_1+1}\zeta_1-1}{q_1-1} e(a_1-1,a_2),& a_1\neq 0\\
 0,& a_1=0
\end{cases}\\
e(a_1,a_2)x_2=\begin{cases}
 (q_1\lambda)^{-a_1}e(a_1,a_2+1),& 0\leq a_2\leq l_2-2\\
 \xi_2(q_1\lambda)^{a_1} e(a_1,0),& a_2=l_2-1
\end{cases}\\
e(a_1,a_2)y_1=\begin{cases}
 e(a_1+1,a_2),& 0\leq a_1\leq l_1-1\\
  \eta_1 e(0,a_2),& a_1=l_1-1
\end{cases}\\
e(a_1,a_2)y_2=\begin{cases}
 \lambda^{a_1}\displaystyle\frac{q_2^{a_2}\zeta_2-\zeta_1}{q_2-1}e(a_1,a_2-1),& a_2\neq 0\\
 \lambda^{a_1}\xi_2^{-1}\displaystyle\frac{\zeta_2-\zeta_1}{q_2-1} e(a_1,l_2-1),& a_2=0
\end{cases}
\end{array}\]  
Thus it follows that $M_2$ is a ${{A}_2^{\underline{q},\Lambda}}$-submodule of $M$. Since $M$ is a simple module, therefore we have $M=M_2$. Here the action of ${{A}_2^{\underline{q},\Lambda}}$ on $M_2$ depends on the scalars $\eta_1,\xi_2,\zeta_1,\zeta_2$. Thus the simple ${{A}_2^{\underline{q},\Lambda}}$-module ${M}_2$ is denoted by $(M_2,\eta_1,\xi_2,\zeta_1,\zeta_2)$ for $\eta_1\in \mathbb{K}$ and $\xi_2,\zeta_1,\zeta_2\in \mathbb{K}^*$ . Now the following result ensures the $\mathbb{K}$-dimension of $M_2$ and the proof is parallel to Theorem \ref{gqa}.
\begin{theo}
The simple ${{A}_2^{\underline{q},\Lambda}}$-module $(M_2,\eta_1,\xi_2,\zeta_1,\zeta_2)$ has dimension $l_1l_2$. Moreover the simple ${{A}_2^{\underline{q},\Lambda}}$-modules $(M_2,\eta_1,\xi_2,\zeta_1,\zeta_2)$ and $(M_2,\eta'_1,\xi'_2,\zeta'_1,\zeta'_2)$ are isomorphic if and only if $\eta_1=\eta'_1,\xi_2=\xi'_2,\zeta_1=q_1^{-a}\zeta'_1$ and $\zeta_2=q_1^{-a}q_2^{b}\zeta'_2$, for some $0\leq a\leq l_1-1$ and $0\leq b\leq l_2-1$.
\end{theo}
Similarly one can easily deal with the following two cases:
\subsubsection{\bf{Consider $\alpha_1\neq 0$ and $\alpha_2=0$}}
Observe that $\kere(x_2)$ is a non-zero subspace of $M$ since $\alpha_2=0$. As the commuting operators (\ref{cev1}) keep the $\mathbb{K}$-space $\kere(x_2)$ invariant, so one can choose a common eigenvector $w\in \kere(x_2)$ of these operators. Put
\[\begin{array}{l}
wx_1^{l_1}=\xi_1 w,\ wx^{l_2}_2=0,\\
wy^{l_i}_i=\eta_iw,~i=1,2\\
wz_i=\zeta_iw,~i=1,2
\end{array}\]
Clearly $\xi_1,\zeta_1,\zeta_2\in\mathbb{K}^*$ and $\eta_1,\eta_2\in\mathbb{K}$. Analogous to the subsection \ref{sov}, here one can easily verify that the set \[S=\{wy_2^{a_2}x_1^{a_1}|0\leq a_i\leq l_i-1\}\] consists of non-zero vectors of $M$. Then the $\mathbb{K}$-subspace $M_3$ of $M$ spanned by the nonzero vectors in $S$ is invariant under ${{A}_2^{\underline{q},\Lambda}}$-action and hence $M=M_3$. Depending on the scalars in module action, the simple ${{A}_2^{\underline{q},\Lambda}}$-module ${M}_3$ is denoted by $(M_3,\xi_1,\eta_2,\zeta_1,\zeta_2)$ for $\eta_1\in \mathbb{K}$ and $\xi_2,\zeta_1,\zeta_2\in \mathbb{K}^*$. The following result provides the $\mathbb{K}$-dimension of $M_3$ and the proof is analogous to Theorem \ref{gqa}.
\begin{theo}
The simple ${{A}_2^{\underline{q},\Lambda}}$-module $(M_3,\xi_1,\eta_2,\zeta_1,\zeta_2)$ has dimension $l_1l_2$. Moreover the simple ${{A}_2^{\underline{q},\Lambda}}$-modules $(M_3,\xi_1,\eta_2,\zeta_1,\zeta_2)$ and $(M_3,\xi'_1,\eta'_2,\zeta'_1,\zeta'_2)$ are isomorphic if and only if $\xi_1=\xi'_1,\eta_2=\eta_2',\zeta_1=q_1^{a}\zeta'_1$ and $\zeta_2=q_1^{a}q_2^{-b}\zeta'_2$, for some $0\leq a\leq l_1-1$ and $0\leq b\leq l_2-1$.
\end{theo}
\subsubsection{\bf{Consider $\alpha_1=0$ and $\alpha_2=0$}} First note that $\kere(x_1)$ and $\kere(x_2)$ are non-zero subspaces of $M$ as $\alpha_1=\alpha_2=0$. Due to the commutation relation $x_1x_2=q_1\lambda x_2x_1$, the nilpotent operator $x_2$ on $M$ must keep the subspace $\kere(x_1)$ invariant. Therefore $\kere(x_1)\cap \kere(x_2)\neq \{0\}$. Now one can choose a common eigenvector $w\in \kere(x_1)\cap \kere(x_2)$ of the commuting operators (\ref{cev1}). This is possible because the operators (\ref{cev1}) keep the $\mathbb{K}$-space $\kere(x_1)\cap \kere(x_2)$ invariant. Take
\[\begin{array}{l}
wx_1=0,\ wx_2=0,\\
wy^{l_i}_i=\eta_iw,~i=1,2\\
wz_i=\zeta_iw,~i=1,2
\end{array}\]
for some $\eta_1,\eta_2\in\mathbb{K}$ and $\zeta_1,\zeta_2\in\mathbb{K}^*$. Now it can be easily verified that the set \[S=\{wy_2^{a_2}y_1^{a_1}|0\leq a_i\leq l_i-1\}\] consists of non-zero vectors of $M$. It follows from the similar arguments used in (\ref{base}), the $\mathbb{K}$-subspace $M_4$ of $M$ spanned by $S$ is invariant under the action of ${{A}_2^{\underline{q},\Lambda}}$ and hence we have $M=M_4$. The simple ${{A}_2^{\underline{q},\Lambda}}$-module ${M}_4$ is denoted by $(M_4,\eta_1,\eta_2,\zeta_1,\zeta_2)$ for $\eta_1,\eta_2\in \mathbb{K}$ and $\zeta_1,\zeta_2\in \mathbb{K}^*$. Now the following result ensures the $\mathbb{K}$-dimension of $M_4$ and the proof is parallel to Theorem \ref{gqa}. 
\begin{theo}
The simple ${{A}_2^{\underline{q},\Lambda}}$-module $(M_4,\eta_1,\eta_2,\zeta_1,\zeta_2)$ has dimension $l_1l_2$. Moreover the simple ${{A}_2^{\underline{q},\Lambda}}$-modules $(M_4,\eta_1,\eta_2,\zeta_1,\zeta_2)$ and $(M_4,\eta'_1,\eta'_2,\zeta'_1,\zeta'_2)$ are isomorphic if and only if $\eta_1=\eta'_1,\eta_2=\eta_2',\zeta_1=q_1^{-a}\zeta'_1$ and $\zeta_2=q_1^{-a}q_2^{-b}\zeta'_2$, for some $0\leq a\leq l_1-1$ and $0\leq b\leq l_2-1$.
\end{theo}
\subsection{Simple $z_1$-torsionfree and $z_2$-torsion ${{A}_2^{\underline{q},\Lambda}}$-modules}\label{sub2}
Suppose $M$ be a $z_1$ torsionfree and $z_2$-torsion simple ${{A}_2^{\underline{q},\Lambda}}$-module. That is $Mz_1\neq 0$ and $Mz_2=0$. Then $M$ becomes a $x_2,y_2$-torsionfree simple module and this will follow from the identity $z_2=z_1+(q_2-1)y_2x_2$. Observe that each of the elements
\begin{equation}\label{cop}
  x^{l_1}_1,y^{l_1}_1,x_2,z_1  
\end{equation} commutes in ${{A}_2^{\underline{q},\Lambda}}$ (see Lemma \ref{crq}). Since $M$ is finite $\mathbb{K}$-dimensional, then there is a common eigenvector $v$ in $M$ of the commuting operators (\ref{cop}). Put 
\[vx_1^{l_1}=\alpha v, \ vy_1^{l_1}=\beta v,\ vx_2=\xi v,\ vz_1=\gamma v,\]
for some $\alpha,\beta,\xi,\gamma\in \mathbb{K}$. It is clear that $\xi$ and $\gamma$ are nonzero. Note that the central elements $x_1^{l_1}$ and $y_1^{l_1}$ act as scalars on $M$, by Schur's lemma. In the following we shall determine the structure of simple ${{A}_2^{\underline{q},\Lambda}}$-module $M$ according to the scalars $\alpha$ and $\beta$:
\subsubsection{\bf{First consider $\alpha\neq 0$}} The operator $x_1$ on $M$ is invertible. So the vectors $vx_1^r$ where $0\leq r\leq l_1-1$ in $M$ are non-zero. Now we claim that the vector subspace $M_5$ spanned by these non-zero vectors is invariant under the action of ${{A}_2^{\underline{q},\Lambda}}$. Indeed it follows easily from the following actions:
\[
\begin{array}{l}
(vx_1^r)x_1=\begin{cases}
vx_1^{r+1},& 0\leq r\leq l_1-2\\
\alpha v,& r=l_1-1
\end{cases}\\
(vx_1^r)y_1=\begin{cases}
 \displaystyle\frac{\gamma q_1^r-1}{q_1-1}vx_1^{r-1},& 1\leq r\leq l_1-1\\
 \alpha^{-1}\displaystyle\frac{\gamma-1}{q_1-1}vx_1^{l-1},& r=0
\end{cases}\\
(vx_1^r)x_2=(q_1\lambda)^r\xi vx_1^r,\ 0\leq r\leq l_1-1\\
(vx_1^r)y_2=\displaystyle\frac{\xi^{-1}\gamma\lambda^{-r}}{1-q_2} vx_1^{r},\ \ 0\leq r\leq l_1-1
\end{array}
\]
Therefore $M$ being a simple module, $M=M_5$. The simple ${{A}_2^{\underline{q},\Lambda}}$-module $M_5$ given above is denoted by $(M_5,\alpha,\xi,\gamma)$ for some $\alpha,\xi,\gamma\in \mathbb{K}^*$. Note that the vectors $vx_1^r$ in $M_5$ are eigenvectors of the operator $z_1$ corresponding to the distinct eigenvalues. Thus we have the following result:
\begin{theo} 
The simple ${{A}_2^{\underline{q},\Lambda}}$-module $(M_5,\alpha,\xi,\gamma)$ has dimension $l_1$. Furthermore the simple ${{A}_2^{\underline{q},\Lambda}}$-module $(M_5,\alpha,\xi,\gamma)$ and $(M_5,\alpha',\xi',\gamma')$ are isomorphic if and only if $\alpha=\alpha', \xi=(q_1\lambda)^r\xi'$ and $\gamma=q^r\gamma'$ for some $0\leq r\leq l_1-1$.
\end{theo}
\subsubsection{\bf{Next consider $\alpha=0$}} One can observe that $x_1$ is nilpotent operator on $M$ and so $\kere(x_1)\neq\{0\}$. Due to the commutation relations with $x_1$, the operators in (\ref{cop}) keep the subspace $\kere(x_1)$ invariant and hence one can choose a common eigenvector $w\in \kere(x_1)$ of the commuting operators (\ref{cop}). Take
\[wx_1=0,\ wy_1^{l_1}=\beta w,\  wx_2=\xi w, wz_1=\gamma w\]
for some $\beta\in\mathbb{K}$ and $\xi,\gamma\in\mathbb{K}^*$. Then it is easily verified that the vectors $wy_1^r,\ 0\leq r\leq l-1$ in $M$ are nonzero. Let $M_6$ be the $\mathbb{K}$-subspace of $M$ spanned by the non-zero vectors $wy_1^r,\ 0\leq r\leq l-1$. One can verify that $M_6$ is invariant under ${{A}_2^{\underline{q},\Lambda}}$-action. In fact after some straightforward computation we have
\[\begin{array}{l}
(wy_1^r)x_1=\begin{cases}
 0,& r=0\\
 \displaystyle\frac{q_1^{-r}-1}{q_1-1}wy_1^{r-1},& 1\leq r\leq l_1-1
\end{cases}\\
(wy_1^r)y_1=\begin{cases}
 wy_1^{r+1},& 0\leq r\leq l_1-2\\
 \beta w,& r=l_1-1
\end{cases}\\
(wy_1^r)x_2=\displaystyle\frac{\xi^{-1}\gamma\lambda^r}{1-q_2}wy_1^r\\
(wy_1^r)y_2=(q_1\lambda)^{-r}\xi wy_1^r 
\end{array}
\]
Since $M$ is a simple module, therefore we have $M=M_6$. The simple ${{A}_2^{\underline{q},\Lambda}}$-module $M_6$ given above is denoted by $(M_6,\beta,\xi,\gamma)$ for some $\beta\in\mathbb{K}$ and $\xi,\gamma\in \mathbb{K}^*$. Finally the $\mathbb{K}$-dimension of $M_6$ is $l_1$ since the vectors $vy_1^r$ are eigenvectors of the operator $z_1$ corresponding to the distinct eigenvalues.
\begin{theo} 
The simple ${{A}_2^{\underline{q},\Lambda}}$-module $(M_6,\beta,\xi,\gamma)$ has dimension $l_1$. Furthermore the simple ${{A}_2^{\underline{q},\Lambda}}$-module $(M_6,\beta,\xi,\gamma)$ and $(M_6,\beta',\xi',\gamma')$ are isomorphic if and only if $\beta=\beta', \xi\xi'(1-q_2)=\lambda^r\gamma'$ and $\gamma=q^{-r}\gamma'$ for some $0\leq r\leq l_1-1$.
\end{theo}
\subsection{Simple $z_1$-torsion ${{A}_2^{\underline{q},\Lambda}}$-modules}\label{subf} Suppose $\mathcal{N}$ be a $z_1$-torsion simple ${{A}_2^{\underline{q},\Lambda}}$-module. Then $\mathcal{N}$ becomes $x_1,y_1$-torsionfree simple module over the factor algebra ${{A}_2^{\underline{q},\Lambda}}/\langle z_1\rangle$ because of the relation $z_1=1+(q_1-1)y_1x_1$. Also this factor algebra ${{A}_2^{\underline{q},\Lambda}}/\langle z_1\rangle$ is isomorphic to the factor $\mathcal{O}_{\mathbf{q}}(\mathbb{K}^4)/\langle (q_1-1)y_1x_1+1\rangle$ of a quantum affine space of rank $4$, where \[\mathbf{q}=\begin{pmatrix}
1&1&\lambda_{12}&q^{-1}_1\lambda^{-1}_{12}\\
1&1&\lambda^{-1}_{12}&q_1\lambda_{12}\\
\lambda^{-1}_{12}&\lambda_{12}&1&q^{-1}_2\\
q_1\lambda_{12}&q^{-1}_1\lambda^{-1}_{12}&q_2&1
\end{pmatrix}.\]
Thus $\mathcal{N}$ is a $x_1,y_1$-torsionfree simple module over $\mathcal{O}_{\mathbf{q}}(\mathbb{K}^4)/\langle (q_1-1)y_1x_1+1\rangle$ as well as over $\mathcal{O}_{\mathbf{q}}(\mathbb{K}^4)$. Now under the assumptions (\ref{asm}) on the defining multiparameter the integral matrix associated to this $\mathbf{q}$ has rank $2$ and then one can easily compute that $\pideg \mathcal{O}_{\mathbf{q}}(\mathbb{K}^4)=\lcmu(l_1,l_2)$ with the help of \cite[Lemma 5.7]{ar}. The simple modules over $\mathcal{O}_{\mathbf{q}}(\mathbb{K}^4)$ have been classified in \cite{smsb}. With this classification, here one can classified such simple module $\mathcal{N}$ and the possible $\mathbb{K}$-dimension of $\mathcal{N}$ is as follows:
\begin{itemize}
    \item if $\mathcal{N}$ is $x_1,y_1,x_2,y_2$-torsionfree, then $\dime_{\mathbb{K}}(\mathcal{N})=\lcmu(l_1,l_2)$.
    \item if $\mathcal{N}$ is $x_1,y_1,x_2$-torsionfree and $y_2$-torsion, then $\dime_{\mathbb{K}}(\mathcal{N})=\ord(q_1\lambda_{12})$. 
    \item if $\mathcal{N}$ is $x_1,y_1,y_2$-torsionfree and $x_2$-torsion, then $\dime_{\mathbb{K}}(\mathcal{N})=\ord(\lambda_{12})$.
   \item if $\mathcal{N}$ is $x_1,y_1$-torsionfree and $x_2,y_2$-torsion, then $\dime_{\mathbb{K}}(\mathcal{N})=1$.
\end{itemize}
\begin{rema}
It is clear from the action of ${{A}_2^{\underline{q},\Lambda}}$ that there does not exist any isomorphism between the above types of simple ${{A}_2^{\underline{q},\Lambda}}$-modules. In view of subsections (\ref{sub1})-(\ref{subf}), one can conclude that the simple $z_1,z_2$-torsion free ${{A}_2^{\underline{q},\Lambda}}$-modules have only maximal $\mathbb{K}$-dimension which is equal to $\pideg ({{A}_2^{\underline{q},\Lambda}})$. Thus we have classified simple ${{A}_2^{\underline{q},\Lambda}}$-modules up to equivalence in terms of scalar parameters.
\end{rema}
\section{\bf{Simple Modules over alternative quantum Weyl algebra}}
In this section we shall focus on classifying simple modules over alternative quantum Weyl algebra of rank $2$ under the assumption (\ref{asm}), analogous to quantum Weyl algebra. Let $N$ be a simple ${\mathcal{A}_2^{\underline{q},\Lambda}}$-module. As pointed out before the $\mathbb{K}$-dimension of $N$ is finite and bounded above by the PI degree $l_1l_2$. Similar to the quantum Weyl algebra, the elements $x_1^{l_1},y_1^{l_1},x_2^l,y_2^l$ are central in alternative quantum Weyl algebra and they act as scalars on $N$, by Schur's lemma. Also the action of each normal element $z_i$ on a simple ${\mathcal{A}_2^{\underline{q},\Lambda}}$-module $N$ is either zero (if $N$ is $z_i$-torsion) or invertible (if $N$ is $z_i$-torsionfree). Now depending on this fact the classification reduces in the following cases.
\subsection{\bf{Simple $z_1,z_2$-torsionfree ${\mathcal{A}_2^{\underline{q},\Lambda}}$-modules}}\label{s1} Suppose $N$ be a $z_1,z_2$-torsionfree simple module over ${\mathcal{A}_2^{\underline{q},\Lambda}}$. Then it follows from Proposition \ref{conn} that $N$ is a $z_1,z_2$-torsionfree simple module over ${{A}_2^{\underline{q},\Lambda}}$. Such simple modules have been classified already in Subsection \ref{sub1}. Here the $\mathbb{K}$-dimension of $N$ is $l_1l_2$.
\subsection{\bf{Simple $z_1$-torsionfree and $z_2$-torsion ${\mathcal{A}_2^{\underline{q},\Lambda}}$-modules}} Suppose $N$ be a $z_1$ torsionfree and $z_2$-torsion simple module over ${\mathcal{A}_2^{\underline{q},\Lambda}}$. Then $N$ becomes a $x_2,y_2$-torsionfree simple module because of the relation $z_2=1+(q_2-1)y_2x_2$ as in Lemma \ref{craq}. Now analogous to the Subsection \ref{sub2}, here one can classify such a simple such simple module $N$ with the commuting operators $x_1^{l_1},y_1^{l_1},x_2,z_2$. Here the $\mathbb{K}$-dimension of $N$ is $l_1$.
\subsection{\bf{Simple $z_1$-torsion and $z_2$-torsionfree ${\mathcal{A}_2^{\underline{q},\Lambda}}$-modules}} Suppose $N$ be a $z_1$-torsion and $z_2$-torsionfree simple ${\mathcal{A}_2^{\underline{q},\Lambda}}$-module. Now it follows from the relation $z_1=1+(q_1-1)y_1x_1$ that $N$ is a $x_1,y_1$-torsionfree simple module. Let $m$ denote the $\lcmu(\ord(\lambda),\ord(q_2))$. Then using the defining relations and Lemma \ref{craq}, one can verify that the elements $x_1,z_2,x_2^m,y_2^m$ commutes in ${\mathcal{A}_2^{\underline{q},\Lambda}}$. Thus here one can classify such a simple module $N$ with these commuting operators. This classification is also analogous to the Subsection \ref{sub2}. Here the $\mathbb{K}$-dimension of $N$ is $m$.
\subsection{\bf{Simple $z_1,z_2$-torsion ${\mathcal{A}_2^{\underline{q},\Lambda}}$-modules}}\label{sf} Suppose $N$ be a $z_1,z_2$-torsion simple ${\mathcal{A}_2^{\underline{q},\Lambda}}$-module. Clearly $N$ becomes $x_1,x_2,y_1,y_2$-torsionfree simple module because of the relations $z_i=1+(q_i-1)y_ix_i$ for $i=1,2$. Finally one can classify such a simple module $N$ with the help of commuting operators $y_1$ and $x_2^r$ where $r=\ord(\lambda)$. Here the $\mathbb{K}$-dimension of $N$ is $r$, where $r=\ord(\lambda)$. 
\par Thus in view of subsections (\ref{s1})-(\ref{sf}), one can conclude that the simple $z_1,z_2$-torsion free ${\mathcal{A}_2^{\underline{q},\Lambda}}$-modules have only maximal $\mathbb{K}$-dimension which is equal to $\pideg ({\mathcal{A}_2^{\underline{q},\Lambda}})$.
\section*{Acknowledgements}
The authors would also like to thank the National Board of Higher Mathematics, Department of Atomic Energy, Government of India for providing financial support to carry out this research.


\begin{thebibliography}{99}
\bibitem {aj} M. Akhavizadegan and D. Jordan, {Prime ideals of quantized Weyl algebras}, {\it Glasgow Math. J.}, {\bf{38}}(3), (1996), 283–297.

\bibitem{ad} J. Alev and F. Dumas, {Rigidit\'e des plongements des quotients primitifs minimaux de $U_q({sl}(2))$ dans lalg\'ebre quantique de Weyl-Hayashi}, {\it Nagoya Math. J.}, {\bf 143}, (1996), 119-146.

\bibitem{bav} V. V. Bavula, {Classification of the simple modules of the quantum Weyl algebra and the quantum plane}, Quantum groups and quantum spaces (Warsaw, 1995), 193–201, Banach Center Publ., 40, Polish Acad. Sci. Inst. Math., Warsaw, 1997.

\bibitem {blt}  J. Boyette , M. Leyk , J. Talley , T. Plunkett and K. Sipe, {Explicit representation theory of the quantum weyl algebra at roots of 1}, {\it Communications in Algebra}, {\bf 28}(11), (2000), 5269-5274.

\bibitem {brg} K. A. Brown and K. R. Goodearl, {\it Lectures on Algebraic Quantum Groups}, Advanced Courses in Mathematics CRM Barcelona, Birkh\"auser Verlag, Basel, 2002.

\bibitem {di} C. De Concini and C. Procesi, {Quantum Groups}, in : {\it D-modules, Representation Theory, and Quantum Groups}, Lecture Notes in Mathematics, 1565, Springer-Verlag, Berlin, 1993, 31-140.

\bibitem{dgo} Y. A. Drozd, B. L. Guzner, and S. A. Ovsienko, {Weight modules over generalized Weyl algebras}, {\it J. Algebra},  {\bf 184}(2), (1996), 491–504.

\bibitem {gz} A. Giaquinto and J. J. Zhang, {Quantum Weyl algebras}, {\it J. Algebra}, {\bf 176}(3), (1995), 861–881.

\bibitem{krg}  K. R. Goodearl, {Prime ideals in skew polynomial rings and quantized Weyl algebras}, {\it J. Algebra}, {\bf 150}(2), (1992), 324–377.

\bibitem{gh2} K. R. Goodearl and  J. T. Hartwig, {The isomorphism problem for multiparameter quantized Weyl algebras}, {\it S\~ao Paulo J. Math. Sci.}, {\bf 9}, (2015), 53-61.

\bibitem {he} B. Heider and L. Wang, {Irreducible Representations of the Quantum Weyl Algebra at Roots of Unity Given by Matrices}, {\it Communications In Algebra}, {\bf 42}(5), (2014), 2156-2162.

\bibitem{dj} D. Jordan, {Finite-dimensional simple modules over certain iterated skew polynomial rings}, {\it J. Pure Appl. Algebra}, {\bf 98}(1), (1995), 45-55.

\bibitem{daj} D. Jordan, {A simple localization of the quantized Weyl algebra}, {\it J. Algebra}, {\bf 174}(1), (1995), 267–281.

\bibitem{lm2} A. Leroy and J. Matczuk, {On q-skew iterated Ore extensions satisfying a polynomial identity}, {\it J. Algebra Appl.}, {\bf 10}(4), (2011), 771–781.

\bibitem{gm} G. Maltsiniotis, {Groupes quantique et structures diff\'erentielles}, {\it C. R. Acad. Sci. Paris S\'er. I Math.}, {\bf 311}, (1990), 831-834.

\bibitem {mcr} J. C. McConnell and J. C. Robson, {\it Noncommutative Noetherian Rings}, Graduate Studies in Mathematics 30, American Mathematical Society, Providence, RI, 2001.

\bibitem{smsb} Mukherjee, S., Bera, S. {Construction of Simple Modules over the Quantum Affine Space.} (To Appear) https://arxiv.org/abs/2001.07432

\bibitem{lr} L. Rigal, {Spectre de l'alg\'ebre de Weyl Quantique}, {\it Beitrage Algebra Geom.}, {\bf 37}, (1996), 119-148.

\bibitem{ar} A. Rogers, {Representations of Quantum Nilpotent Algebras at Roots of unity and their completely prime quotients}, PhD Thesis, University of Kent, 2019.

\bibitem{xt} X. Tang, {Automorphisms for Some Symmetric Multiparameter Quantized Weyl Algebras and Their Localizations}, {\it Algebra Colloq.}, {\bf 24}(3), (2017), 419–438.



\end{thebibliography}
\end{document}